\documentclass[12pt]{amsart}
\usepackage{latexsym,amssymb,amsfonts,amsmath, amscd, euscript, MnSymbol}
\usepackage{pdfsync}
\usepackage{enumerate}
\usepackage[T1]{fontenc}
\usepackage[pdftex]{graphicx}
\usepackage{wrapfig}
\usepackage{epsfig}
\usepackage{pgf,tikz}
\usetikzlibrary{arrows,automata}
\usetikzlibrary{positioning,calc}
\usepackage{amssymb, amsmath, amsthm, enumerate}
\usepackage{mathrsfs}
\usepackage{url}
\usepackage{epstopdf}
\usepackage{enumerate}
\usepackage{verbatim}
\usepackage{float}
\usepackage{layout}
\usepackage[top=3cm, bottom=3cm, left=2cm, right=2cm]{geometry}
\usepackage{mathtools}

\newtheorem{theorem}{Theorem}[section]

\newtheorem{conjecture}[theorem]{Conjecture}
\newtheorem{corollary}[theorem]{Corollary}

\newtheorem{definition}[theorem]{Definition}
\newtheorem{example}[theorem]{Example}

\newtheorem{remark}[theorem]{Remark}

\newtheorem{question}[theorem]{Question}
\newtheorem{lemma}[theorem]{Lemma}

\usepackage{siunitx}
\usepackage{graphics,mwe}

\newsavebox{\hold}
\newlength{\holdht}

\begin{document}

\title[Best proximity pairs in ultrametric spaces]{Best proximity pairs in ultrametric spaces}
\author[K. Chaira, O. Dovgoshey and S. Lazaiz]{Karim Chaira$^1$, Oleksiy Dovgoshey$^2$ and Samih Lazaiz$^{3}$}
\address{$^1$Ben M'sik Faculty of Sciences \\ Hassan II University, Casablanca, Morocco}
\email{chaira\_karim@yahoo.fr}
\address{$^2$Institute of Applied Mathematics and Mechanics of NASU\\ Dobrovolskogo str. 1, Slovyansk 84100, Ukraine}
\email{oleksiy.dovgoshey@gmail.com}
\address{$^3$ENSAM Casablanca, Hassan II University, Casablanca, Morocco}
\email{samih.lazaiz@gmail.com}


\begin{abstract}
In the present paper, we study the existence of best proximity pairs in ultrametric spaces. We show, under suitable assumptions, that the proximinal pair $(A,B)$ has a best proximity pair. As a consequence we generalize a well known best approximation result and we derive some fixed point theorems. Moreover, we provide examples to illustrate the obtained results.

\end{abstract}

\maketitle

\section{Introduction and preliminaries}

Let $F:A\rightarrow B$ be a map where $A$ and $B$ are two nonempty subsets of a metric space $M$. Recall that a point $x\in M$ is called a fixed point of a $F$ if $F(x)=x$. It is known that such an equation does not always have a solution. However, in the absence of the fixed point (for example if $A\cap B=\varnothing$), it is possible to consider the problem of finding a point $x\in A$ that is as close as possible to $F(x)$ in $B$; i.e., to minimize the quantity $d(x,F(x))$ over $A$. This type of problem is considered as part of approximation theory, more specifically best approximation point results. 

\begin{definition}\cite{ref12}\label{def11}
Let $(M, d)$ be a metric space. A subset $A\subset M$ is said to be proximinal if given any $x\in M$ there exists $a_0\in A$ such that 
$$d(x,a_0)=dist(x,A)=\inf\{d(x,z):\; z\in A\}.$$
Such an $a_0$, if it exists, is called a best approximation to $x$ in $A$.
\end{definition}

In the literature, some positive results concerning the existence of the best approximation points were given whether in the archimedean or non-archimedean spaces, for more details see \cite{ref9,ref10,ref11}. 

\medskip 
Recall that, for subsets $A$ and $B$ of a metric space $(M,d)$, we set
$$dist(A,B)=\inf\{d(a,b):\;a\in A\;\text{and}\;b\in B\},$$
and write $\delta(A)$ for the diameter of $A$, i.e., $\delta(A):=\sup\{d(x,y)\; : \;x,y\in A \}.$

\medskip 
It should be noted here that if $A$ and $B$ are proximinal subsets of a metric space $(M, d)$, then, in general, there is no reason to have $d(x, y) = dist(A, B)$ for $x \in A$ and $y \in B$ (see Example \ref{CEx} below). This led the authors in \cite{ref12} to introduce the notion of the best proximity pair.

\begin{definition}\cite[Definition 1.1]{ref12}\label{def12}
Let $M$ be a metric space and let $A$ and $B$ be nonempty subsets of $M$. Let
$$A_0=\{x\in A:\; d(x, y)=dist(A, B)\; \text{for some}\; y\in B\};$$
$$B_0=\{y\in B:\; d(x, y)=dist(A, B)\; \text{for some}\; x\in A\}.$$

A pair $(x, y) \in A_0 \times B_0$ for which $d(x, y)=dist(A, B)$ is called a best proximity pair for the sets $A$ and $B$.
\end{definition} 

Let $A$ and $B$ be subsets of a metric space $M$. We will say that the pair $(A, B)$ is proximinal if $A$ and $B$ are proximinal. The following question naturally arises.

\begin{question}\label{Q1}
Let $(A,B)$ be a proximinal pair of $(M,d)$. Does there exists a best proximity pair $(a,b)\in A_0\times B_0$ ? If so, is the pair $(A_0,B_0)$ also proximinal ?
\end{question}

If this question has an affirmative answer, this inspires to formulate the following:

\begin{question}\label{Q2}
Given a mapping $F:A\cup B \rightarrow A\cup B$ with $F(A)\subset A$ and $F(B)\subset B$ (also called noncyclic mapping), does there exists an ordered pair $(a,b)\in A\times B$ such that 
$$
F(a)=a, \quad F(b)=b\quad \text{and}\quad d(a,b)=dist(A,B).
$$
\end{question}

There is an extensive literature contains partial affirmative answers to these two questions in the context of metric spaces and linear spaces (see \cite{ref13,ref14,ref15}). To the best knowledge of the authors, this is the first time these questions are considered in the case of ultrametric spaces. 

\medskip
Recall that an ultrametric space is a metric space $M$ in which strong triangle inequality $d(x,y) \leq \max\{d(x,z),d(z, y)\}$ is satisfied for all $x,y,z \in M$. 

\begin{remark}
It is interesting to note that an axiomatic characterization of proximity spaces generated by ultrametrics was obtained by A. Yu. Lemin in 1984 (see \cite{ref200}).
\end{remark}

\begin{definition}\cite{ref22}
An ultrametric space $(M,d)$ is called spherically complete if each nested sequence of closed balls $B_1 \supset  B_2 \supset\cdots$  has a nonempty intersection.
\end{definition}

The spherical complete ultrametric spaces were first introduced by Ingleton \cite{ref23} in order to obtain an analog of the Hanh-Banach theorem for non-Archimedean valued fields. This notion has numerous applications in studies of fixed point results for ultrametric spaces \cite{ref24,ref3,ref17}. It was shown by Bayod and Mart\`inez-Maurica \cite{ref25} that an ultrametric space is spherically complete if and only if this space is ultrametrically injective. Recall that an ultrametric space $(Y,\rho)$ is ultrametrically injective if for each $F : A \rightarrow X$, where $A \subseteq Y$ and $X$ is a space with an ultrametric $d$, the condition
$$d(F (x), F (y)) \leq\rho(x, y),\qquad \forall x, y \in A$$

implies the existence of an extension $\tilde{F}: Y \rightarrow X$ of the mapping $F$ such that
$$d(\tilde{F} (x), \tilde{F} (y)) \leq\rho(x, y),\qquad \forall x, y \in Y.$$

Thus, an ultrametric space is ultrametrically injective if every contractive mapping from this space to arbitrary ultrametric space has a contractive extension. Some interesting results related to spherical completeness of ultrametric spaces can also be found in \cite{ref26} and \cite{ref27}.

\medskip
In \cite{ref8} the authors prove the following result.

\begin{theorem}\label{KS1}
  Let $A$ be a nonempty spherically complete subspace of an ultrametric space $M$. Then $A$ is proximinal in $M$.
\end{theorem}

The next definition slightly narrows the corresponding definition from \cite{ref3}.

\begin{definition}\rm\label{Def18}
Let $(M,d)$ be an ultrametric space. Assume that $F:M\rightarrow M$ is a map and $B=B(x,r)$, $r>0$, is a closed ball in $(M,d)$. We say that $B$ is \textit{a minimal $F$-invariant ball} if : 
\begin{itemize}
  \item[(i)] $F(B)\subseteq B$, and
  \item[(ii)] $d(y,F(y))=r$ for each $y\in B$. 
\end{itemize}
\end{definition}

\begin{remark}
Definition \ref{Def18} implies, in particular, that any minimal $F$-invariant ball $B$ cannot have the form $\{p\}$, where $p$ is a fixed point of $F$ and $\{p\}$ is the one-point set consisting the only element $p$.
\end{remark}

The following definition is well known.

\begin{definition}
Let $A\subseteq M$. $F:A\rightarrow A$ is said to be nonexpansive if  
\[
d(F(x),F(y))\leq d(x,y)\quad \text{for all}\; x,y\in A
\]
\end{definition}

\begin{remark}\label{Rem19}
Let $(M, d)$ be a spherically complete ultrametric space, let $F : M \rightarrow M$ be non-expensive, and let $\mathbf{B}_T$ be the set of all closed balls $B$ which are $F$-invariant, 
$$F(B) \subseteq B.$$
Then $B\in\mathbf{B}_T$ is minimal $F$-invariant if and only if $B_1 \subseteq B$ implies $B_1 = B$ for every $B_1 \in \mathbf{B}_T$.
\end{remark}

The next theorem is one of the most significant results of the fixed point theory in ultrametric spaces.

\begin{theorem}\label{KS}\cite{ref3,ref2}
Suppose $M$ is a spherically complete ultrametric space and $F:M\rightarrow M$ is a nonexpansive map. Then every closed ball 
$$B(x, d(x, T (x))) = \{y \in M : d(x, y) \leq d(x, T (x))\}$$
contains either a fixed point of $F$ or a minimal $F$-invariant ball. 
\end{theorem}

Therefore, to guarantee the existence of a fixed point of $F : M \rightarrow M$, it is necessary to add some additional restrictions. For example, we can consider strictly contractive mappings instead of contractive one.

\begin{definition}
Let $(M,d)$ be an ultrametric space and $F : M \rightarrow M$ a mapping. We say that :
\begin{enumerate}
  \item $F$ is strictly contractive if $d(F(x), F(y)) < d ( x , y )$ whenever $x\neq y$.
  \item $F$ is strictly contractive on orbit if $F(x)\neq x$ implies $d( F^2x, F(x)) < d ( F(x), x )$ for each $x \in M$.
\end{enumerate}
\end{definition} 

Using the following theorem, we can prove the existence of fixed points for every strictly contractive on orbit mapping $F : M \rightarrow M$ of each spherically complete ultrametric space $(M, d)$. 

\begin{theorem}\label{w-regular}
Let $(M,d)$ be a spherically complete ultrametric space.  Let $F:M\rightarrow M$ be a nonexpansive map satisfying 
\begin{equation}\label{eq1}
  \liminf_{n\rightarrow \infty}\ d(F^n(x),F^{n+1}(x))< d(x,F(x))
\end{equation}
whenever $x$ in $M$ and $x \neq F(x)$.  Then $F$ has a fixed point in any $F$-invariant closed ball.
\end{theorem}

\begin{proof}
Suppose that $B$ is a $F$-invariant closed ball which does not contain a fixed point of $F$, then, by Theorem \ref{KS}, there are $r > 0$ and $x \in B$ such that
$$0<r=d(x,F(x))=d(F(x),F^2(x))=\ldots=d(F^n(x),F^{n+1}(x))< d(x,F(x)) $$
contrary to (\ref{eq1}).
\end{proof}

It should be noted here that introduced in \cite{ref18} the mappings $F : M \rightarrow M$ satisfying the weak regular property are also satisfy conditions of Theorem \ref{w-regular}.

\begin{conjecture}
Let $(M, d)$ be a spherically complete ultrametric space and let a mapping $F:M\rightarrow M$ be nonexpansive. Then the following conditions are equivalent:
\begin{itemize}
  \item[(i)] $F$ has a fixed point in every nonempty spherically complete subspace of $M$.
  \item[(ii)] Inequality (\ref{eq1}) holds whenever $x\in M$ and $x\neq F (x)$.
\end{itemize}
\end{conjecture}

\medskip
In this paper, we will see that the answer to the above questions is positive under natural assumptions. In the lack of this assumptions, we show by an example that $(A_0,B_0)$ may be an empty pair. As a consequence we generalize a best approximation result due to Kirk and Shahzad (see \cite[Theorem 11]{ref3}) and we derive some fixed point theorems. Throughout this paper, we provide some examples to illustrate the obtained results. Our main theorems continue and strenghten the corresponding results from \cite{CL}.

\section{Main results}

We first give a useful lemma.

\begin{lemma}\label{Lemma1}
Let $(A,B)$ be a proximinal pair of a nonempty ultrametric space $(M, d)$ and let the inequality $\delta(B)\leq dist(A,B)$ hold. Then $A_0$  is nonempty and the equality $B_0= B$ holds.
\end{lemma}

\begin{proof}
Let us consider arbitrary $a \in A$ and $b, b^\prime \in B$. We claim that the equality
\begin{equation}\label{eqq1}
  d(a,b)=d(a,b^\prime)
\end{equation}
holds. Indeed, the strong triangle inequality and the inequality $\delta(B)\leq dist(A,B)$ imply that
$$
\begin{array}{ccl}
  d(a,b^\prime) & \leq & \max\{d(a,b);d(b,b^\prime) \}\\
                & \leq & \max\{d(a,b);\delta(B) \}\\
                & \leq & \max\{d(a,b);dist(A,B) \}\\
                & \leq & d(a,b).
\end{array}
$$

Thus, $d(a,b^\prime)\leq d(a, b)$ holds. Similarly, we obtain $d(a, b) \leq d(a,b^\prime)$. Equality (\ref{eqq1}) follows.

\medskip
Now we can easily prove that $B_0 = B$ and $A_0 \neq\varnothing$. Indeed, let $b^\prime$ be an arbitrary point of $B$. Then, using equality (\ref{eqq1}), we obtain

\begin{equation}\label{eq2}
  dist(A, B) =\inf_{a\in A,b\in B} d(a, b) = \inf_{a\in A} d(a, b^\prime) = dist(b^\prime, A).
\end{equation}
Since $A$ is proximinal, there is $a^\prime\in A$ such that $dist(b^\prime , A) = d(a^\prime, b^\prime)$. The last equality
and (\ref{eq2}) imply
$$d(a^\prime, b^\prime)=dist(A,B).$$
Thus, for every $b^\prime\in B$ there is $a^\prime \in A$ such that (\ref{eqq1}) holds. It implies $A_0\neq\varnothing$ and $B_0 = B$.
\end{proof}

Next, we give an example to show that if $\delta(B)>dist(A,B)$, the pair $(A_0,B_0)$ may be an empty pair.

\begin{example}\label{CEx}
Let $M=\mathbb{N}_0$ be the set of positive integers and define the ultrametric distance $d$ on $M$ as follows:
\[
d(n,m)=\begin{cases}
  \begin{array}{ll}
    0 & \text{if}\; n=m\\
    \max\{\frac{1}{n},\frac{1}{m}\} & \text{otherwise.}
  \end{array}
\end{cases}
\]
Then $(M,d)$ is an ultrametric space. Set $A=2\mathbb{N}_0=\{2,4,\ldots\}$ and $B=2\mathbb{N}_0-1=\{1,3,\ldots\}$. 

\medskip
It is clear that $A$ and $B$ are proximinal and the equalities $\delta(B)=1$, $\delta(A)=\frac{1}{2}$ and $dist(A,B)=0$ hold. Now note that 
$$dist(A,B)\leq\min\{\delta(B),\delta(A)\}$$ 
and $A_0=B_0=\varnothing$.
\end{example}

\begin{remark}
The ultrametric that we use in Example \ref{CEx} was apparently first constructed by Delhommé, Laflamme, Pouzet and Sauer \cite[Proposition 2]{ref201}. This construction is very often useful in the study of various topological and geometrical properties
of ultrametric spaces \cite{ref202,ref203,ref204,ref205} and has a natural generalization to Priess-Crampe and Ribenboim Ultrametrics with totally ordered range sets (see \cite[Proposition 4.10]{ref206}).
\end{remark}

\begin{lemma}\label{Lemma23}
Let $A$ and $B$ satisfy the conditions of Lemma \ref{Lemma1} and let $b_0$ be an arbitrary point of $B$. Write $r := dist(A, B)$. Then we have
\begin{equation}\label{eq3}
  A_0=A\cap S(b_0,r),
\end{equation}
where
\begin{equation}\label{eq4}
  S(b_0 , r) := \{x \in M : d(b_0 , x) = r\}
\end{equation}
is the sphere with the radius $r$ and the center $b_0$.
\end{lemma}
\begin{proof}
Suppose that $r = 0$. Then the inequality $\delta(B)\leq dist(A,B)$ implies $B = \{b_0\}$, where $\{b_0\}$ is the set containing the point $b_0$ only. Hence, by Definition \ref{def12},
\begin{equation}\label{eq5}
  A_0 = \{x \in A : d(x, b_0 ) = 0\}
\end{equation}
holds. Since $A_0$ is nonempty, from (\ref{eq5}) it follows that $A_0 =\{b_0\}$. The last equality and equality (\ref{eq4}) imply (\ref{eq3}) for $r = 0$.

\medskip
For $r > 0$ we consider the open ball 
$$B^\prime(b_0 , r) := \{x \in M : d(b_0 , x) < r\}$$
and the "exterior" $E(b_0 , r)$ of this ball,
$$E(b_0 , r) := \{x \in M : d(b_0 , x) > r\}.$$

It is clear that $M = B^\prime(b_0 , r) \cup S(b_0 , r) \cup E(b_0 , r)$. Consequently, the set $A_0$ can be represented as
$$A_0=A_0\cap M= [A_0\cap B^\prime(b_0 , r)] \cup [A_0\cap S(b_0 , r)] \cup [A_0\cap E(b_0 , r)].$$
Thus, equality (\ref{eq3}) holds if and only if
\begin{equation}\label{eq6}
  A_0\cap B^\prime(b_0 , r)=\varnothing=A_0\cap E(b_0 , r).
\end{equation}
To prove (\ref{eq6}), we first note that
$$d(a,b_0)\geq dist(a,B)\geq dist(A,B)=r$$
holds for every $a \in A$. Consequently, the intersection $A\cap B^\prime(b_0 , r)$ is empty,
\begin{equation}\label{eq7}
  A\cap B^\prime(b_0 , r)=\varnothing
\end{equation}
It implies $A_0\cap B^\prime(b_0 , r)=\varnothing$ because $A_0 \subseteq A$.

\medskip
If $A_0\cap E(b_0 , r)\neq\varnothing$, then there is $a_0\in A_0$ such that
\begin{equation}\label{eq8}
  d(a_0 , b_0 ) = r_1 > r = dist(A, B).
\end{equation}

It follows from the proof of Lemma \ref{Lemma1} (see (\ref{eqq1})) that
\begin{equation}\label{eq9}
  d(a_0 , b) = d(a_0 , b_0 )
\end{equation}
holds for every $b \in B$. Since $a_0 \in A_0$, there is $b^\prime\in B$ such that $d(b^\prime , a_0 ) = dist(A, B)$ by
Definition \ref{def12}. Using (\ref{eq8}), the last equality and equality (\ref{eq9}) with $b^\prime= b$, we obtain the
contradiction,
$$d(a_0,b_0)<dist(A,B)=d(b^\prime,a_0)=d(a_0,b_0),$$
that implies the second equality in (\ref{eq6}).
\end{proof}

The next lemma follows from \cite[Proposition 18.5]{ref26}.
\begin{lemma}\label{Lemma24}
Let $(M, d)$ be an ultrametric space and let $B$ be a ball in $(M, d)$. Then
\begin{equation}\label{eq10}
  d(x, b) = d(x, a)
\end{equation}
holds for every $x \in M \setminus B$ and all $a, b \in B$.
\end{lemma} 

The following theorem provides a partial answer to Question \ref{Q1}.

\begin{theorem}\label{Thm25}
  Let $(A, B)$ be a proximinal pair in a nonempty ultrametric space $(M, d)$. Then the following statements are equivalent:
  \begin{itemize}
    \item[(i)] There is a point $a^\prime\in A$ such that
\begin{equation}\label{eq11}
d(a^\prime, b) = dist(A, B)  
\end{equation}
for every $b \in B$.
    \item[(ii)] The inequality $\delta(B)\leq dist(A, B)$ holds.
    \item[(iii)] The sets $A_0$ and $B_0$ are proximinal subsets of $(M, d)$, and the equalities $B_0 = B$ and $dist(A_0 , B_0 ) = dist(A, B)$ hold, and every $(a, b) \in A_0 \times B_0$ is a best proximity pair for both $(A, B)$ and $(A_0 , B_0 )$.         
  \end{itemize}
\end{theorem}

\begin{proof}
(i)$\Rightarrow$(ii). Let (\ref{eq11}) hold with a fixed $a^\prime\in A$ and all $b \in B$. Then, for arbitrary $b_1 , b_2 \in B$, the  strong triangle inequality and (\ref{eq11}) imply
$$d(b_1,b_2)\leq dist(A,B).$$
Statement (ii) follows.

\medskip
(ii)$\Rightarrow$(iii). Let $\delta(B)\leq dist(A, B)$ hold. Let us prove that $(A_0 , B_0 )$ is proximinal. By Lemma \ref{Lemma1}, the equality $B_0 = B$ holds. Thus, $B_0$ is proximinal. Hence, it suffices to show that $A_0$ is proximinal.

Let $b_0$ be an arbitrary point of $B_0$. By Lemma \ref{Lemma23}, we have
\begin{equation}\label{eq12}
  A_0 = A \cap S(b_0 , r)
\end{equation}
with $r = dist(A, B)$. Let us consider the closed ball $B(b_0 , r)$ with the center $b_0$ and the radius $r$, 
\begin{equation}\label{eq13}
  B(b_0 , r) = \{x\in M : d(b_0 , x) \leq r\} = B^\prime(b_0 , r) \cup S(b_0 , r)
\end{equation}
Using (\ref{eq12}), (\ref{eq13}) and equality (\ref{eq7}), we can represent the set $A_0$ as
\begin{equation}\label{eq14}
  A_0 = A \cap B(b_0 , r)
\end{equation}
By Definition \ref{def11}, $A_0$ is proximinal if for every $x^\prime\in M$ there is $a_0 \in A_0$ such that $d(x^\prime , a_0 ) = dist(x^\prime , A_0 )$. Let $x^\prime$ be an arbitrary point of $M$. If $x^\prime\notin B(b_0 , r)$, then
\begin{equation}\label{eq15}
  d(b_0,x^\prime)>r
\end{equation}
holds. Using (\ref{eq14}), (\ref{eq15}) and Lemma \ref{Lemma24}, we obtain that
\begin{equation}\label{eq16}
  d(a_0,x^\prime)=d(b_0,x^\prime)>r
\end{equation}
holds for every $a_0\in A_0$. Consequently, the equality 
\begin{equation}\label{eq17}
  d(x^\prime,a_0)=dist(x^\prime,A_0)
\end{equation}
holds for every $a_0\in A_0$.

Suppose now that $x^\prime \in B(b_0 , r)$. Since $A$ is proximinal, there exists $a^\prime\in A$ such that
\begin{equation}\label{eq18}
  d(x^\prime,a^\prime)=dist(x^\prime,A)
\end{equation}
It is easy to see that $a^\prime \in B(b_0 , r)$. Indeed, if $a^\prime\notin B(b_0 , r)$, then, similarly to (\ref{eq16}), we
obtain $d(x^\prime , a^\prime ) > r$.

Now if $a_0$ is an arbitrary point of $A_0$, then from (\ref{eq14}) and $x^\prime\in B(b_0 , r)$ it follows that $d(x^\prime , a_0 ) \leq r$. (The diameter of any ultrametric ball is less than or equal to its radius see \cite[Proposition 1.2]{ref16}.) Thus, we have the contradiction,
$$r < d(x^\prime, a^\prime ) = dist(x^\prime , A) \leq dist(x^\prime , A_0 ) \leq dist(x^\prime , a_0 ) \leq r$$
and, consequently, the equality $d(x^\prime , a^\prime ) = dist(x^\prime, A_0 )$ holds. Thus, $A_0$ is proximinal.

To complete the proof of validity of the implication (ii)$\Rightarrow$(iii), it suffices to note that
$$d(a_0 , b_0 ) = dist(A, B)$$
holds for every pair $(a_0 , b_0 )\in A_0 \times B_0$ by Lemma \ref{Lemma23}.

\medskip
(iii)$\Rightarrow$(i). If (iii) holds, then (\ref{eq11}) yields with any $a^\prime\in A_0$ for each $b\in B$ because $B_0 = B$.
\end{proof}

\medskip
Theorem \ref{Thm25} and Lemma \ref{Lemma23} completely describe the structure of the set of all best proximity pairs for $A$ and $B$ if $\delta(B) \leq dist(A, B)$. It seems to be interesting to find a generalization of these results for the case when the proximinal pair $(A, B)$ is arbitrary.

\medskip
Since every spherically complete set is proximinal, and every point of each ultrametric ball is a center of this ball, and every closed ball in spherically complete space is spherically complete, formula (\ref{eq14}) and Theorem \ref{Thm25} imply the following

\begin{corollary}\label{Cor27}
Let $(A, B)$ be a nonempty spherically complete pair in an ultrametric space $(M, d)$. If $\delta(B)\leq dist(A, B)$ holds, then $(A_0 , B_0 )$ is also a nonempty spherically complete pair in $(M, d)$.
\end{corollary}

The next theorem partially answers Question \ref{Q2} for nonexpansive mappings. 

\begin{theorem}\label{Thm1}
Let $A$ and $B$ be nonempty spherically complete sets in an ultrametric space $(M, d)$ and let $\delta(B) \leq dist(A, B)$ hold. Suppose $F: A \cup B \rightarrow A \cup B$ is a noncyclic nonexpansive mapping. Then there exists a best proximity pair $(a^*, b^* ) \in A\times B$ satisfying exactly one of the following statements:
  \begin{itemize}
    \item[(i)] $a^*$ and $b^*$ are fixed points of $F$.
    \item[(ii)] $a^*$ is a fixed point of $F$, and $B(b^*,d(b^*,F(b^*)))$ is a minimal $F$-invariant ball in $B$, each point of which is a nearest point to $a^*$.
    \item[(iii)] $b^*$ is a fixed point of $F$ in $B$, and $B(a^*,d(a^*,F(a^*)))$ is a minimal $F$-invariant ball in $A$, and each point of which is a nearest point to $b^*$.
    \item[(iv)] $B(a^*,d(a^*,F(a^*)))$ and $B(b^*,d(b^*,F(b^*)))$ is a minimal $F$-invariant balls in $A$ and, respectively, in $B$, and in addition, every pair 
    $$(x,y)\in B(a^*,d(a^*,F(a^*)))\times B(b^*,d(b^*,F(b^*)))$$
    is a a best proximity pair for $B(a^*,d(a^*,F(a^*)))$ and $B(b^*,d(b^*,F(b^*)))$
  \end{itemize}  
\end{theorem}

\begin{proof}

It follows from Definition \ref{Def18} that Statements (i)--(iv) are pairwise inconsistent (see Remark \ref{Rem19}). Therefore, it is enough to find a best proximity pair $(a^*, b^* ) \in A\times B$ for which at least one of statements (i)--(iv) is fulfilled.

\medskip
By Corollary \ref{Cor27}, $A_0$ and $B_0$ are nonempty spherically complete subsets of $(M, d)$. Let $x\in A_0$ and $y \in B_0$ be arbitrary. Then, using the implication (ii)$\Rightarrow$(iii) from Theorem \ref{Thm25},
we obtain the equality
\begin{equation}\label{eqv420}
  dist(A, B) = d(x, y).
\end{equation}

Since $F$ is nonexpansive, equality (\ref{eqv420}) implies the inequality

\begin{equation}\label{eqv421}
d(F(x), F (y)) \leq dist(A, B).
\end{equation}

The mapping $F$ is noncyclic, thus, we have $F(x) \in A$ and $F(y) \in B$. Consequently, the inequality $dist(A, B) \leq d(F(x), F (y))$ holds. The last inequality and (\ref{eqv421}) imply
$$d(F(x), F(y)) = dist(A, B).$$

Thus, we have $(F(x), F(y)) \in A_0 \times B_0$ for every $(x, y) \in A_0 \times B_0$. It implies the inclusions $F(A_0 ) \subseteq A_0$ and $F(B_0 ) \subseteq B_0$.

\medskip
If $dist(A, B) = 0$ holds, then from $\delta(B) \leq dist(A, B)$ it follows that $B$ is a single-point set. Hence, there is a unique $b^*\in M$ such that $B = B_0 = \{b^*\}$. Now the inclusion $F(B_0 ) \subseteq B_0$ gives us $F(b^* ) = b^*$. Similarly, by Lemma \ref{Lemma23}, for the case $dist(A, B) = 0$, we can find a unique $a^*\in A$ such that $A_0 = \{a^*\}$. The last inequality and the inclusion $F(A_0 ) \subseteq A_0$ imply $F(a^*) = a^*$. Thus, Statement (i) holds if $dist(A, B) = 0$.

\medskip
Let us consider the case when $dist(A, B) > 0$. Using Statement (iii) of Theorem \ref{Thm25} and applying Theorem \ref{KS} to the mappings $F_{|A_0}$ and $F_{|B_0}$, where $F_{|A_0}$  and $F_{|B_0}$ are the restrictions of $F$ on $A_0$ and, respectively, on $B_0$, we see that the theorem is true when $A_0 = A$.

\medskip
To complete the proof, it suffices to show that every closed ball in $A_0$ is a closed ball in $A$. To see it, we note that
\begin{equation}\label{eqv422}
  A_0 = A \cap S(b_0 , r)
\end{equation}
holds by Lemma \ref{Lemma23}, when $S(b_0 , r) = \{x \in M : d(x_0 , x) = r\}$, $b_0 \in B$ and $r = dist(A, B)$.
Moreover, by (\ref{eq6}), we have the equality
\begin{equation}\label{eqv423}
  A_0 \cap B^\prime(b_0 , r)=\varnothing
\end{equation}
for $B^\prime(b_0 , r) = \{x\in M : d(b_0 , x) < r\}$, where $b_0$ and $r$ are the same as in (\ref{eqv422}). Now
equalities (\ref{eqv422}) and (\ref{eqv423}) imply
$$A_0 = A \cap B(b_0 , r),$$
i.e., $A_0$ is a closed ball in $A$. Since $A$ is a subspace of ultrametric space $(M, d)$ and $A_0 \subseteq A$ holds, every closed ball in the closed ball $A_0$ is also a closed ball in $A$.

\end{proof}

The following example shows that Theorem \ref{Thm1} cannot be strengthen by removing any of Statements (i)--(iv).

\begin{example}\label{ex29}
Let $A = \{a_1 , a_2\}$ and $B = \{b_1 , b_2\}$ be disjoint sets and let $d$ be an ultrametric on $A \cup B$ such that $d(a_1 , a_2 ) = d(b_1 , b_2 ) = 1$ and $d(a, b) = 2$ whenever $a \in A$ and $b \in B$. Let us consider the permutations

\begin{equation*}
  F_1=\begin{pmatrix}
a_1 & a_2 & b_1 & b_2 \\
a_1 & a_2 & b_1 & b_2 
\end{pmatrix}, \qquad F_2=\begin{pmatrix}
a_1 & a_2 & b_1 & b_2 \\
a_1 & a_2 & b_2 & b_1 
\end{pmatrix},
\end{equation*}
\begin{equation*}
  F_3=\begin{pmatrix}
a_1 & a_2 & b_1 & b_2 \\
a_2 & a_1 & b_1 & b_2 
\end{pmatrix}, \qquad F_4=\begin{pmatrix}
a_1 & a_2 & b_1 & b_2 \\
a_2 & a_1 & b_2 & b_1 
\end{pmatrix}.
\end{equation*}
Then every $F_i,$  $i = 1,\ldots, 4$, is a noncyclic nonexpansive mapping. In addition, $F_1$ satisfies Statement (i) but does not satisfy any of statements (ii), (iii) and (iv). $F_2$ satisfies Statement (ii) but does not satisfy any of statements (i), (iii) and (iv), and so on.
\end{example}

\begin{conjecture}
Statements (i)--(iv) from Theorem \ref{Thm1} remain valid for all nonempty spherically complete sets $A$ and $B$ even if the inequality $\delta(B) \leq dist(A, B)$ does not hold.
\end{conjecture}

\begin{corollary}\cite[Theorem 11]{ref3}
Let $A$ be a nonempty spherically complete subspace of an ultrametric space $M$, and let $b^* \in M \setminus A$. Suppose $F : M \rightarrow M$ is a mapping for which $F(b^*) = b^*$. Also assume that $F$ is nonexpansive on $A \cup \{ b^* \}$ and that $A$ is $F$-invariant. Then $F$ has a fixed point in $A$ which is a nearest point of $b^*$ in $A$, or $A$ contains a minimal $F$-invariant set, each point of which is a nearest point to $b^*$ in $A$.
\end{corollary}

Next, we derive some future fixed point results partially answering Question \ref{Q2}.

\begin{theorem}\label{Thm2}
 Let $A$ and $B$ be nonempty spherically complete subsets of an ultrametric space $(M, d)$ and let $\delta(B) \leq dist(A, B)$ hold. Suppose $F: A \cup B \rightarrow A \cup B$ is a noncyclic nonexpansive mapping satisfying
  $$\liminf_{n\rightarrow \infty}\ d(F^n(x),F^{n+1}(x))< d(x,F(x))$$

whenever $x$ in $M$ and $x \neq F(x)$. Then, there exist $a\in A$ and $b\in B$ such that
$$F(a)=a, \quad F(b)=b\quad \text{and}\quad d(a,b)=dist(A,B).$$
\end{theorem}

\begin{proof}
By Corollary \ref{Cor27} $A_0$ and $B_0$ also are nonempty spherically complete subsets of $(M, d)$. As in the proof of Theorem \ref{Thm1}, we obtain that $F(A_0 ) \subseteq A_0$ and $F(B_0 ) \subseteq B_0$.

\medskip
Now the existence of $a \in A_0$ and $b \in B_0$ which satisfy $F(a) = a$ and $F(b) = b$ follows from Theorem \ref{w-regular}. The pair $(A, B)$ is proximinal by Theorem \ref{KS1}. Using Theorem \ref{Thm25}, we obtain $d(a, b) = dist(A, B)$ since $a \in A_0$ and $b \in B_0$.

\end{proof}

Since every strictly contractive mapping $F: M \rightarrow M$ is nonexpansive and has a unique fixed point if $M$ is spherically complete (see \cite{ref19}), we obtain the following.

\begin{theorem}\label{Thm3}
Let $A$ and $B$ be nonempty spherically complete subspaces of an ultrametric space $(M, d)$ and let $\delta(B)\leq dist(A,B)$ hold. If there is a noncyclic strictly
contractive $F: A \cup B \rightarrow A \cup B$, then this $F$ has a unique fixed point $p$ and the equalities
\begin{equation}\label{eqv424}
  B = B_0 = A_0 = \{p\}
\end{equation}
hold.
\end{theorem}

\begin{proof}
Suppose that there is a noncyclic strictly contractive mapping $F : A \cup B \rightarrow A \cup B$. Then $F_{|A}$ and $F_{|B}$, the restrictions of $F$ on $A$ and, respectively, on $B$, are also strictly contractive. Since $A$ and $B$ are nonempty and spherically complete, we can find a unique $a^*\in A$ and a unique $b^*\in B$ such that  
\begin{equation}\label{eqv425}
  F_{|A}(a^*)=a^*\quad\text{and}\quad F_{|B}(b^*)=b^*.
\end{equation}
As noted in the proof of Theorem \ref{Thm1}, the inclusions $F(A_0 ) \subseteq A_0$ and $F(B_0 ) \subseteq B_0$ hold. Moreover, $A_0$ and $B_0$ are also nonempty and spherically complete by Corollary \ref{Cor27}. Hence, we can find the unique $a_0 \in  A_0$ and $b_0 \in  B_0$ such that
\begin{equation}\label{eqv426}
  F_{|A_0}(a_0)=a_0\quad\text{and}\quad F_{|B_0}(b_0)=b_0.
\end{equation}
Now the inclusions $A_0 \subseteq A$, $B_0 \subseteq B$, equalities (\ref{eqv425}), (\ref{eqv426}), and the uniqueness of $a^*$ and $b^*$ satisfying (\ref{eqv425}) imply the equalities $a^*= a_0$ and $b^*= b_0$. If $a^*$ and $b^*$ are distinct, then
\begin{equation}\label{eqv427}
  d(F(a_0),F(b_0))<d(a_0,b_0).
\end{equation}
because $F$ is strictly contractive. Using Theorems \ref{KS1} and \ref{Thm25}, and inequality (\ref{eqv427}), we obtain the contradiction,
$$dist(A, B) = d(a_0 , b_0 ) > d(F(a_0 ), F(b_0 )) > dist(A, B).$$
Thus, the point $b_0$ is a unique fixed point of $F : A \cup B \rightarrow A \cup B$ and $dist(A, B) = 0$ holds. To complete the proof, it suffices to note that (\ref{eqv424}) with $p = b_0$ now follows from $a_0 \in A_0$ and $b_0 \in B_0$ by Theorems \ref{KS1} and \ref{Thm25}.
\end{proof}

As a consequence of Theorem \ref{Thm3}, we obtain the following corollary that strengths Corollary 12 from \cite{ref3}.

\begin{corollary}
Let $A$ be a nonempty spherically complete subspace of an ultrametric space $M$ and let $F : M \rightarrow M$ be a mapping having a fixed point $b^* \in M$. Assume that $F$ is strictly contractive on $A \cup \{ b^*\}$ and $A$ is $F$-invariant. Then $b^*$ is a point in $A$.
\end{corollary}

\bibliographystyle{plain}

\end{document}